\documentclass[11pt, a4paper]{amsart}
\usepackage{fullpage}

\usepackage{amsfonts}
\usepackage{caption}
\usepackage{float}
\usepackage{graphicx}
\usepackage{amssymb}
\usepackage{amsmath, amsthm,color}
\usepackage{hyperref}
\usepackage{pdfsync}
\usepackage{bbm, dsfont}
\usepackage[margin=0.96in]{geometry}

\hypersetup{ocgcolorlinks=true,allcolors=testc}
\hypersetup{
     colorlinks   = true,
     citecolor    = black
}
\hypersetup{linkcolor=black}
\hypersetup{urlcolor=black}

\newcommand{\R}{\mathbb{R}}

\newtheorem{theorem}{Theorem}

\newtheorem{lemma}[theorem]{Lemma}
\newtheorem{question}[theorem]{Question}
\newtheorem{proposition}[theorem]{Proposition}

\newtheorem{remark}[theorem]{Remark}

\newcommand*\diff{\mathop{}\!\mathrm{d}}

\begin{document}
\title{Dimension independent Bernstein--Markov inequalities\\ in Gauss space}
\author{Alexandros Eskenazis}
\author{Paata Ivanisvili}
\thanks{This work was carried out under the auspices of the Simons Algorithms and Geometry (A\&G) Think Tank.}
\address{Department of Mathematics, Princeton University}
\email{ ae3@math.princeton.edu \textrm{(A.\ Eskenazis)}}

\address{Department of Mathematics, Princeton University; UC Irvine. }
\email{paatai@math.princeton.edu \textrm{(P.\ Ivanisvili)}}


\begin{abstract} 
We obtain the following dimension independent Bernstein--Markov inequality in Gauss space: for each $1\leq p<\infty$ there exists a constant $C_p>0$ such that for any $k\geq 1$ and all polynomials $P$ on $\mathbb{R}^{k}$ we have 
$$
\| \nabla P\|_{L^{p}(\mathbb{R}^{k}, \diff\gamma_k)} \leq C_p (\mathrm{deg}\, P)^{\frac{1}{2}+\frac{1}{\pi}\arctan\left(\frac{|p-2|}{2\sqrt{p-1}}\right)}\|P\|_{L^{p}(\mathbb{R}^{k}, \diff\gamma_k)}, 
$$
where $\diff\gamma_k$ is the standard Gaussian measure on $\mathbb{R}^{k}$. We also show that under some mild growth assumptions on any  function $B \in C^{2}((0,\infty))\cap C([0,\infty))$ with $B', B''>0$ we have 
$$
\int_{\mathbb{R}^{k}} B\left( |LP(x)|\right) \diff\gamma_k(x) \leq \int_{\mathbb{R}^{k}} B\left( 10 (\mathrm{deg}P)^{\alpha_{B}}|P(x)|\right)\diff\gamma_k(x)
$$
where $L=\Delta-x\cdot \nabla $ is the generator of the Ornstein--Uhlenbeck semigroup and 
$$
\alpha_{B} =1+\frac{2}{\pi} \arctan\left(\frac{1}{2}\sqrt{\sup_{s \in (0,\infty)}\left\{\frac{sB''(s)}{B'(s)}+\frac{B'(s)}{sB''(s)}\right\}-2}\right).
$$
\end{abstract}
\maketitle 
{\footnotesize

}

\section{Introduction}
Let $\diff\gamma_{k}(x)$ be the standard Gaussian measure on $\mathbb{R}^{k}$,  given by
\begin{align*}
\diff\gamma_{k}(x) = \frac{e^{-|x|^{2}/2}}{\sqrt{(2\pi)^{k}}} \diff x,
\end{align*}
where $|x| = \sqrt{x_1^2+\cdots+x_k^2}$ is the Euclidean length of $x=(x_1,\ldots,x_k)\in\mathbb{R}^k$. Here and throughout, we will denote by $\varphi_k$ the density of the Gaussian measure $\diff\gamma_k$ with respect to the Lebesgue measure on $\R^k$. For $1\leq p < \infty$, define $L^{p}(\mathbb{R}^{k}, \diff\gamma_k)$ to be the space of those measurable functions on $\mathbb{R}^{k}$ for which 
\begin{align*}
\|f\|_{L^{p}(\mathbb{R}^{k}, \diff\gamma_k)} :=\Big(\int_{\mathbb{R}^{k}} |f|^{p} \diff\gamma _k\Big)^{\frac{1}{p}}<\infty. 
\end{align*}
As usual, $L^\infty(\mathbb{R}^k,\diff\gamma_k)$ is defined by the condition $\|f\|_{L^{\infty}(\mathbb{R}^{k}, \diff\gamma_k)} = \mathrm{esssup}_{x \in \mathbb{R}^{k}} |f(x)| <\infty$. For convenience of notation, we will abbreviate $\|f\|_{L^{p}(\mathbb{R}^{k}, \diff\gamma_k)}$ as $\|f\|_{L^{p}(\diff\gamma_k)}$.

\subsection{Freud's inequality in high dimensions}

In his seminal paper \cite{Fr1}, Freud obtained the following weighted Bernstein--Markov type inequality on the real line.
\begin{theorem}[Freud's inequality, \cite{Fr1}] \label{thm:Freudold}
There exists a universal constant $C>0$ such that for any $1\leq p \leq \infty$ and all polynomials $P$ on $\mathbb{R}$, we have 
\begin{align}\label{froid}
\left(\int_{\mathbb{R}} | P'(x) \varphi_1(x)|^{p} \diff x\right)^{1/p} \leq C \sqrt{\mathrm{deg}P} \left( \int_{\mathbb{R}} | P(x) \varphi_1(x)|^{p} \diff x\right)^{1/p}.
\end{align}
\end{theorem}
After making a change of variables in (\ref{froid}), Freud's inequality can be rewritten in terms of $\|\cdot \|_{L^{p}(\diff\gamma_1)}$ norms as 
\begin{align} \label{gaza1}
\|P'\|_{L^{p}(\diff\gamma_1)} \leq C \sqrt{\frac{\mathrm{deg} P}{p}} \| P\|_{L^{p}(\diff\gamma_1)},
\end{align}
for all $1\leq p<\infty$. Notice that (\ref{gaza1}) breaks down for $p=\infty$ as $\|P\|_{L^\infty(\diff\gamma_1)}=\infty$ for every non-constant polynomial $P$, nevertheless inequality (\ref{froid}) still persists. 

After proving Theorem \ref{thm:Freudold}, Freud~\cite{Fr2} extended his Gaussian estimates \eqref{froid} to more general weights $e^{-Q(x)}$ on the real line, nowadays known as {\em Freud weights}, where the function $Q(x)$ satisfies certain growth and convexity assumptions. In this case, the bound $\sqrt{\mathrm{deg}P}$ in (\ref{froid}) is replaced by a certain quantity which depends on the so-called  {\em Mhaskar--Rakhmanov--Saff numbers} of the weight $e^{-Q(x)}$. Since the works \cite{Fr1,Fr2} of Freud, several different proofs of such one-dimensional weighted Bernstein--Markov inequalities have been found (see, e.g., \cite{FrNev,NevTot,LevLub,LevLub2,LevLub3,KrooSza}), in part due to important implications of such estimates in approximation theory (see, e.g., \cite[Theorem~2]{Fr1} and \cite[Theorems 4.1 and 5.1]{Fr2}). We refer the reader to the beautiful survey \cite{Lub1} of Lubinsky for a detailed exposition of results on this subject. 

In relation to the ``heat smoothing conjecture'' \cite{MN1} one can ask if a dimension independent discrete counterpart of Freud's inequality holds on the Hamming cube $\{-1,1\}^{n}$ equipped with uniform counting measure~\cite{EI}. A positive answer by central limit theorem would imply the validity of Freud's inequality in $L^{p}(\mathbb{R}^{k}, \diff\gamma_{k})$ with constants independent of $k$. Therefore, it is of interest first to understand if Freud's inequality can be extended to higher dimensions with a dimensionless constant.   Throughout the ensuing discussion, for a smooth function $f:\mathbb{R}^k\to \mathbb{R}$ and $0<p<\infty$, we will denote
\begin{equation}
\|\nabla f\|_{L^p(\diff\gamma_k)} := \Big( \int_{\mathbb{R}^k} \Big(\sum_{j=1}^k (\partial_j f)^2(x) \Big)^{p/2} \diff\gamma_k(x)\Big)^{1/p}.
\end{equation}
We first notice (see also Section \ref{mult}) that the case $p=\infty$ of Freud's inequality \eqref{froid} easily extends in all $\mathbb{R}^k$ with a constant independent of the dimension.

\begin{proposition}\label{best}
There exists a universal constant $C>0$ such that for any $k\geq 1$, and all polynomials $P$ on $\mathbb{R}^{k}$ we have 
\begin{align}
\|\varphi_k \nabla P \|_{L^\infty(\mathbb{R}^k)} \leq C \sqrt{\mathrm{deg}\, P}\,  \|\varphi_k P \|_{L^\infty(\mathbb{R}^k)},
\end{align}
where $\mathrm{deg}\, P$ denotes the total degree of the multivariate polynomial $P$. 
\end{proposition}

For finite values of $p$, the following question naturally arises, in analogy to \eqref{gaza1}.

\begin{question}[Bernstein--Markov inequality in Gauss space] \label{question}
 Is it true that for each $1\leq p<\infty$  there exists a constant $C_p>0$ such that for any integer $k\geq 1$, and all polynomials $P$ on $\mathbb{R}^{k}$ the the dimension independent Gaussian Bernstein--Markov inequality
\begin{align}\label{markg}
\|\nabla P\|_{L^{p}(\diff\gamma_k)} \leq C_p \sqrt{\mathrm{deg}\, P} \|P\|_{L^{p}(\diff\gamma_k)}
\end{align}
holds true?
\end{question}

\begin{remark}
Using \eqref{gaza1}, it is straightforward to obtain \eqref{markg} with a dimension dependent constant $C_{p,k}$. Also, inequality \eqref{markg} can easily be proven for $p=2$ (and $C_2=1$) by expanding $P$ in the Hermite basis and using orthogonality.
\end{remark}


Before moving to our main result, we mention that an elegant argument of Maurey and Pisier from \cite{MP1}, implies a weakening of Question \ref{question} with a (suboptimal) linear bound on ${\mathrm{deg}\, P}$.

\begin{proposition}\label{rot1}
There exists a universal constant $C>0$ such that for any $k\geq 1$, any $ 0< p < \infty$  and  all polynomials $P$ on $\mathbb{R}^{k}$, we have 
\begin{align}\label{trick1}
\| \nabla P\|_{L^{p}(\diff\gamma_k)} \leq C \frac{\mathrm{deg}\, P}{\sqrt{p+1}} \|P\|_{L^{p}(\diff\gamma_k)}.
\end{align}
\end{proposition}

The main result of the present paper is that the linear bound on $\mathrm{deg}\,P$ in (\ref{trick1}) can be improved.
\begin{theorem}\label{mth02}
For each $1< p<\infty$ there exists a constant $C_p>0$ such that for any $k\geq 1$, and  all polynomials $P$ on $\mathbb{R}^{k}$, we have 
\begin{align}\label{mth01}
\| \nabla P\|_{L^{p}(\diff\gamma_k)} \leq C_p  (\mathrm{deg}\, P)^{\frac{1}{2}+\frac{1}{\pi}\arctan\left(\frac{|p-2|}{2\sqrt{p-1}}\right)}\|P\|_{L^{p}(\diff\gamma_k)}.
\end{align}
\end{theorem} 
Notice that for each $p \in (1, \infty)$ we have  $0\leq\frac{1}{\pi}\arctan\left(\frac{|p-2|}{2\sqrt{p-1}}\right)<\frac{1}{2}$, therefore (\ref{mth01}) is worse than (\ref{markg}) but improves upon (\ref{trick1}). Also, notice that for $p=2$, inequality (\ref{mth01}) recovers (\ref{markg}). To the extend of our knowledge, these are the best known bounds towards Question~\ref{markg}. 

 Our proof of Theorem \ref{mth02}, relies on a similar Bernstein--Markov type inequality for the generator of the Ornstein--Uhlenbeck semigroup (see Theorem \ref{mth03}  below) and Meyer's dimension-free Riesz transform inequalities in Gauss space from \cite{M1}.

\subsection{Reverse Bernstein--Markov inequality in Gauss space}
Our initial motivation to study Question \ref{question} comes from a dual question that Mendel and Naor \cite[Remark~5.5 (2)]{MN1} asked on the Hamming cube. A positive answer to their question would, by standard considerations, imply its continuous counterpart in Gauss space, namely a {\em reverse Bernstein--Markov inequality}. To state the latter question precisely, 
let $H_{m}$ be the probabilists'  Hermite polynomial of degree $m$ on $\mathbb{R}$, i.e.,
\begin{align}  \label{hermite1}
H_{m}(s) = \int_{\mathbb{R}}(s+it)^{m} \diff\gamma_1(t).
\end{align}
For $x = (x_{1}, \ldots, x_{k})\in \mathbb{R}^{k}$ and a multiindex $\alpha = (\alpha_{1}, \ldots, \alpha_{k})$, where $\alpha_{j} \in \mathbb{N}\cup\{0\}$, we consider the multivariate Hermite polynomial on $\R^k$, given by
\begin{align} \label{hermite2}
H_{\alpha}(x) = \prod_{j=1}^{k}H_{\alpha_{j}}(x_{j}).
\end{align}
The family $\{ H_{\alpha}\}_{\alpha}$ forms an orthogonal system on $L^{2}(\diff\gamma_k)$. Denote by $|\alpha|=\alpha_{1}+\ldots+ \alpha_{k}$ and let $L=\Delta-x \cdot \nabla$ be the generator of the  Ornstein--Uhlenbeck semigroup. Then, one has 
$$
LH_{\alpha}(x) = -|\alpha| H_{\alpha}(x).
$$
for every multiindex $\alpha$. The operator $L$ should be understood as the {\em Laplacian in Gauss space}. Now consider any polynomial $P$ on $\mathbb{R}^{k}$ which {\em lives on frequencies greater than $d$}, i.e., of the form
\begin{equation} \label{hf1}
P(x) = \sum_{|\alpha|\geq d} c_\alpha H_\alpha(x),
\end{equation}
where $c_\alpha\in\mathbb{C}$.
\begin{question} [Mendel--Naor, \cite{MN1}] \label{qmennao}
Is it true that for each $1<p<\infty$ there exists a constant $c_p>0$ such that for any $k\geq 1$, any $d \geq 1$, and all polynomials of the form (\ref{hf1}) on $\mathbb{R}^k$, living on frequencies greater than $d$, we have 
\begin{align}\label{revm}
\| L P\|_{L^{p}(\diff\gamma_k)} \geq c_p d \|P\|_{L^{p}(\diff\gamma_k)}.
\end{align}
\end{question}

In \cite{EI}, we show that for every $1<p<\infty$ there exists some $c_p>0$ such that for all polynomials $P$ which live on frequencies $[d,d+m]$, i.e. are of the form
$$P(x) = \sum_{d\leq|\alpha|\leq d+m} c_\alpha H_\alpha(x),$$
we have
\begin{equation} \label{eq:ei}
\|LP\|_{L^p(\diff\gamma_k)} \geq c_p \frac{d}{m} \|P\|_{L^p(\diff\gamma_k)}.
\end{equation}
For small values of $m$, \eqref{eq:ei} improves upon previously known bounds in Question \ref{qmennao} which follow from works of Meyer \cite[Lemma~5.4]{MN1} and Mendel and Naor \cite[Theorem~5.10]{MN1} on the Hamming cube for this smaller subclass of polynomials. In particular, when $m=O(1)$, \eqref{eq:ei} positively answers a special case of Question \ref{qmennao}. We refer to \cite{EI} for further results on reverse Bernstein--Markov inequalities along with extensions for vector-valued functions on the Hamming cube.

\subsection{Bernstein--Markov inequality with respect to $L$} In order to prove Theorem \ref{mth02}, it will be convenient to first study the analogue of Question \ref{question} for the ``second derivative" $L$, namely, is it true that for every polynomial $P$ on $\mathbb{R}^k$, we have
\begin{equation}
\| L P\|_{L^{p}(\diff\gamma_k)} \leq  C_p \, \mathrm{deg} P\,  \|P\|_{L^{p}(\diff\gamma_k)}\ ?
\end{equation}
The best result that we could obtain in this direction is the following theorem.

\begin{theorem}\label{mth03}
For any integer $k\geq 1$, any $p\geq 1$, and any polynomial $P$ on $\mathbb{R}^{k}$, we have 
\begin{align}\label{Lmarkov}
\| L P\|_{L^{p}(\diff\gamma_k)} \leq 10  (\mathrm{deg}\, P)^{1+\frac{2}{\pi}\arctan\left(\frac{|p-2|}{2\sqrt{p-1}}\right)}\|P\|_{L^{p}(\diff\gamma_k)}.
\end{align}
\end{theorem}  

\subsection{General function estimates}
Our techniques for proving Theorem \ref{mth03} allow us to replace $p$-th powers in $L^{p}$ norms in (\ref{Lmarkov}) by an arbitrary convex increasing function in the spirit of Zygmund's theorem~(see \cite{zyg}, Vol 2., Ch. 10, Theorem~(3.16)). We recall that Zygmund's theorem asserts that if $\Phi$ is nondecreasing convex function on $[0,\infty)$, and 
$$
f_{n}(t) = \frac{a_{0}}{2}+\sum_{k=1}^{n}(a_{k} \cos(kt)+b_{k} \sin(kt))
$$
is a trigonometric polynomial of degree at most $n$, then for every $r\geq 1$, the sharp inequality 
\begin{align}\label{arestoval}
\int_{0}^{2\pi}\Phi(|f^{(r)}_{n}(t)|)\diff t \leq \int_{0}^{2\pi}\Phi(n^{r}|f_{n}(t)|)\diff t
\end{align}
holds true. In fact,  by the results of Arestov~\cite{Ar1, Ar2}, the inequality holds for somewhat larger class of nondecreasing  functions, i.e.,  $\Phi(t) = \psi(\ln t)$ for some convex $\psi$ on $(-\infty, 
\infty)$. In particular, inequality (\ref{arestoval}) holds true for  $\Phi(t)=t^{p}$ for every $p>0$ (instead of just $p\geq1$), thus implying the usual $L^p$ Bernstein--Markov inequality for trigonometric polynomials. 

One straightforward way to obtain an analog of (\ref{arestoval}) in Gauss space is to invoke the rotational invariance of the Gaussian measure. Indeed, we will shortly show the following estimates.

\begin{theorem}\label{rot2}
Let $\Phi:[0,\infty)\to\mathbb{R}$ be an increasing convex function. For any $k\geq 1$, and all polynomials $P$ on $\mathbb{R}^{k}$ we have 
\begin{align}
&\int_{\mathbb{R}^{k}} \int_{\mathbb{R}} \Phi\big(|t\nabla P(x)|\big) \diff\gamma_1(t) \diff\gamma_k(x) \leq \int_{\mathbb{R}^{k}}  \Phi\big((\mathrm{deg}P)\,|P(x)|\big) \diff\gamma_k(x);\label{ph1}
\end{align}
and
\begin{align}
&\int_{\mathbb{R}^{k}}  \Phi\big(|LP(x)|\big) \diff\gamma_k(x) \leq \int_{\mathbb{R}^{k}}  \Phi\big((\mathrm{deg}P)^{2}|P(x)|\big) \diff\gamma_k(x).\label{ph2}
\end{align}
\end{theorem}

Our main result of this section is that under mild assumptions on $\Phi$, one can further \mbox{improve (\ref{ph2}).}

\begin{theorem}\label{mth04}
For any $k\geq 1$ and all polynomials $P$ on $\mathbb{R}^{k}$, we have 
\begin{align}
\int_{\mathbb{R}^{k}} B\big( |LP(x)|\big) \diff\gamma_k(x) \leq \int_{\mathbb{R}^{k}} B\big( 10 (\mathrm{deg}P)^{\alpha_{B}}|P(x)|\big)\diff\gamma_k(x)
\end{align} 
for any function $B \in C([0,\infty))\cap C^{2}((0,\infty))$ with $B', B''>0$, such that for every $x>0$
$$\max\{|B(x+\varepsilon)|,B'(x+\varepsilon), B''(x+\varepsilon)\} < C(1+x^{2N})$$
for each fixed $\varepsilon>0$ and some  $C=C(\varepsilon), N=N(\varepsilon)>0$. Here 
\begin{align}
\alpha_{B} :=1+\frac{2}{\pi} \arctan\left(\frac{1}{2}\sqrt{\sup_{s \in (0,\infty)}\left\{\frac{sB''(s)}{B'(s)}+\frac{B'(s)}{sB''(s)}\right\}-2}\right).
\end{align}
\end{theorem}

A straightforward optimization shows that Theorem~\ref{mth03} follows from Theorem~\ref{mth04} by considering $B(t)=t^p$, where $p\geq1$.

The rest of the paper is structured as follows. In Section \ref{sec2}, we present the proof of Theorem~\ref{rot2} and its consequence, Proposition~\ref{rot1}. In Section \ref{sec3}, we prove our main result, Theorem \ref{mth04} from which we also deduce Theorem \ref{mth03}. Finally, Section \ref{sec4} contains the deduction of Theorem \ref{mth02} from Theorem \ref{mth03} and Section \ref{mult} contains the proof of Proposition \ref{best}.

\subsection*{Acknowledgments}

We would like to thank Volodymyr Andriyevskyy for discussions  that helped us to simplify the proof of Lemma~\ref{Gabor1}. We are also thankful to Piotr Nayar and the anonymous referees for helpful comments and suggestions. P. I.  was partially supported by NSF DMS-1856486 and NSF CAREER-1945102. 

\section{Proof of Theorem~\ref{rot2}} \label{sec2}

We first prove Theorem~\ref{rot2}. The argument is inspired by an idea of Maurey and Pisier \cite{MP1} and only uses the rotational invariance of the Gaussian measure $\diff\gamma_k$  and Zygmund's inequality \eqref{arestoval}.

\medskip

\noindent {\it Proof of Theorem~\ref{rot2}.} Let 
$$
X=(X_{1}, \ldots, X_{k}) \quad \text{and} \quad Y=(Y_{1}, \ldots, Y_{k})
$$ 
be two independent multivariate standard Gaussian $\mathcal{N}(0,Id_{k})$ random variables on $\mathbb{R}^{k}$. Take any polynomial $P(x)$ on $\mathbb{R}^{k}$ of degree $n$, and consider the trigonometric polynomial 
\begin{align*}
t(\theta) :=P(X\cos(\theta) + Y\sin(\theta)).
\end{align*}
Clearly $\mathrm{deg}(t(\theta))\leq n$. It follows from Zygmund's inequality (\ref{arestoval}) that  
\begin{align}\label{bers1}
\int_{0}^{2\pi} \Phi(|t'(\theta)|) \diff\theta \leq  \int_{0}^{2\pi} \Phi(n |t(\theta)|) \diff\theta.
\end{align}
Since $t'(\theta) = \nabla P(X \cos(\theta)+Y \sin(\theta))\cdot (-X\sin(\theta)+Y\cos(\theta))$, and the random variables 
$$
X \cos(\theta)+Y \sin(\theta) \quad \text{and} \quad  -X\sin(\theta)+Y\cos(\theta)
$$
 are also independent multivariate standard Gaussians, we see that after taking the expectation of (\ref{bers1}) with respect to $X,Y$, we obtain 
\begin{align*}
2 \pi \mathbb{E} \Phi(|\nabla P(X)| |Y_1|) & = \int_{0}^{2\pi}\mathbb{E} \Phi(|\nabla P(X \cos(\theta)+Y \sin(\theta))\cdot (-X\sin(\theta)+Y\cos(\theta))|) \diff\theta \\
&\leq \int_{0}^{2\pi}\mathbb{E} \Phi(n |P(X\cos(\theta)+Y \sin(\theta))|) \diff\theta = 2 \pi   \mathbb{E}\Phi(|P(X)|).
\end{align*}
This finishes the proof of (\ref{ph1}). To prove (\ref{ph2}) notice that 
\begin{align*}
t''(\theta) & =\langle \mathrm{Hess}P(X\cos(\theta) + Y\sin(\theta)) (-X\sin(\theta)+Y\cos(\theta)), (-X\sin(\theta)+Y\cos(\theta))\rangle\\
& - \nabla P(X \cos(\theta)+Y \sin(\theta))\cdot (X \cos(\theta)+Y \sin(\theta)), 
\end{align*}
where $\langle \cdot, \cdot \rangle$ denotes inner product in $\mathbb{R}^{k}$.
Therefore following the same steps as before and  using Zygmund's inequality (\ref{arestoval}), we obtain 
\begin{align*}
  \mathbb{E} \Phi(|\langle\mathrm{Hess}P(X)Y,Y\rangle -\nabla P(X)\cdot X|)\leq  \mathbb{E} \Phi(n^{2}|P(X)|).
\end{align*}
Finally, it remains to use convexity of the map $x \mapsto \Phi(|x|)$ and Jensen's inequality  to get
\begin{align*}
\mathbb{E}_{Y}\Phi(|\langle\mathrm{Hess}P(X)Y,Y\rangle -\nabla P(X)\cdot X|) \geq \Phi(|\mathbb{E}_{Y}(\langle\mathrm{Hess}P(X)Y,Y\rangle -\nabla P(X)\cdot X)|)=\Phi(|LP(X)|),
\end{align*}
which completes the proof of (\ref{ph2}). \hfill$\Box$

\medskip

Deriving Proposition \ref{rot1} is now straightforward.

\medskip

\noindent {\it Proof of Proposition \ref{rot1}.} It follows from the proof of Theorem \ref{rot2} that \eqref{ph1} holds true under the sole assumption $\Phi(t) = \psi(\ln t)$ with $\psi$ nondecreasing and convex on $(-\infty, \infty)$. Thus, applying \eqref{ph1} to $\Phi(t)=t^p$, where $p>0$, we deduce that
\begin{equation} \label{eq:useth7}
\big(\mathbb{E} |g|^p\big)^{1/p} \|\nabla P\|_{L^p(\diff\gamma_k)} \leq \mathrm{deg}\, P \|P\|_{L^p(\diff\gamma_k)},
\end{equation}
where $g$ is a standard Gaussian random variable. Inequality \eqref{trick1} then follows from \eqref{eq:useth7} since $\big(\mathbb{E} |g|^p\big)^{1/p} = \sqrt{2} \Big( \frac{\Gamma\big(\frac{p+1}{2}\big)}{\sqrt{\pi}} \Big)^{1/p} \asymp \sqrt{p+1}$ for $p>0$.
\hfill$\Box$

\section{Proof of Theorem~\ref{mth04}} \label{sec3}

In this section, we prove the main result of this paper, Theorem~\ref{mth04}, and its consequence, the Bernstein--Markov inequality of Theorem~\ref{mth03}.

\subsection{Step 1. A general complex hypercontractivity}
Any polynomial $P(x)$  on $\mathbb{R}^{k}$ admits a representation of the form 
\begin{equation}
P(x) = \sum_{|\alpha|\leq \mathrm{deg}(P)} c_{\alpha}H_{\alpha}(x),
\end{equation}
for some coefficients $c_\alpha\in\mathbb{C}$. Next, given $z \in \mathbb{C}$, we  define the action of the second quantization operator (or Mehler transform) $T_{z}$ on $P(x)$ as  
\begin{align}
T_{z}P(x) = \sum_{|\alpha|\leq \mathrm{deg}(P)} z^{|\alpha|}c_{\alpha}H_{\alpha}(x).
\end{align}
Clearly $T_{0}P(x) = c_0 = \int_{\mathbb{R}^{k}}P(x) \diff\gamma_k(x)$, and $T_{1}P(x)=P(x)$. 

\medskip

In what follows we will be working with a real-valued function $R \in C^{2}((0,\infty)) \cap C([0,\infty))$ (and sometimes we will further require $R \in C^{2}([0,\infty))$) such that 
\begin{align}\label{assump1}
|R(x)|, |R'(x)|, |R''(x)| \leq C(1+x^{N})
\end{align}
for some constants $C,N>0$ and every $x\geq0$. These assumptions are sufficient  to avoid integrability issues. 

\begin{lemma} \label{lem:computation}
Fix $z \in \mathbb{C}$,  and let $R \in C^{2}([0,\infty))$ be a real-valued function such that $R'\geq 0$.  Assume that 
\begin{align}\label{inf1}
(1-|z|^{2}) R'(x)|w|^{2}+2xR''(x)((\Re w)^{2}-(\Re zw)^{2})\geq 0 \quad \text{for all} \quad w \in \mathbb{C} \mbox{ and } x \geq 0.
\end{align}
Then, for all $k\geq 1$, and for all polynomials $P(x)$ on $\mathbb{R}^{k}$ we have 
\begin{align}\label{hyp1}
\int_{\mathbb{R}^{k}} R(|T_{z}P(x)|^{2})\diff\gamma_k(x) \leq \int_{\mathbb{R}^{k}} R(|P(x)|^{2})\diff\gamma_k(x).
\end{align}
\end{lemma}
\begin{remark}
We will see later (see Section~\ref{neces}) that the reverse implication also holds, i.e., inequality (\ref{hyp1}) implies (\ref{inf1}) under the additional assumption that $R \in C^{3}((0, \infty))$. 
\end{remark}
\begin{proof}
 Denote the scaled Gaussian measure on $\mathbb{R}^{k}$ of variance $s$ by
\begin{align*}
\diff\gamma_{k}^{(s)}(x) =\frac{1}{\sqrt{(2\pi s)^{k}}} e^{-|x|^{2}/2s}\diff x,\quad  s \in (0,1].
\end{align*}
Take any polynomial $g : \mathbb{R}^{k} \to \mathbb{R}$. We will denote partial derivatives by lower indices, for example 
$$
g_{x_{i}x_{j}}(x) := \frac{\partial^{2}}{\partial x_{i} \partial x_{j}} g(x).
$$

Fix a complex number $z \in \mathbb{C}$ satisfying (\ref{inf1}),  and consider the map 
\begin{align}\label{flow1}
g(x,u,s) := \int_{\mathbb{R}^{k}} \int_{\mathbb{R}^{k}} g((u+iv)+z(x+iy))\diff\gamma_{k}^{(s)}(v) \diff\gamma_k^{(1-s)}(y),
\end{align}
where for $x=(x_{1}, \ldots, x_{k}) \in \mathbb{R}^{k}$ we denote $zx=(z x_{1}, \ldots, zx_{n})$.  An analysis done in \cite{IV, HU, janson} suggests that one should study the monotonicity of the following map,

\begin{align}\label{izrdeba}
[0,1]\ni s  \longmapsto r(s) = \int_{\mathbb{R}^{k}}\int_{\mathbb{R}^{k}}R(|g(x,u,s)|^{2}) \diff\gamma_k^{(s)}(u) \diff\gamma_k^{(1-s)}(x), 
\end{align}
where $\diff\gamma_k^{(s)}(u)$ at $s=0$ (or $ \diff\gamma_k^{(1-s)}(x)$ at $s=1$) should be understood as delta measure at zero, i.e.,  
$$
\lim_{s \to 0} \int_{\mathbb{R}^{k}} R(|g(x,u,s,)|^{2}) \diff\gamma_k^{(s)}(u) \stackrel{u= \tilde{u}\sqrt{s}}{=} \lim_{s \to 0}\int_{\mathbb{R}^{k}} R(|g(x,\tilde{u} \sqrt{s},s,)|^{2}) \diff\gamma_k (\tilde{u})= R(|g(x,0,0)|^{2})
$$ by Lebesgue's dominated convergence theorem.

First notice that if $s \mapsto r(s)$ is increasing then \eqref{hyp1} follows. Indeed,  consider any polynomial on $\mathbb{R}^k$ of the form $P(x)=\sum_{\alpha\leq\mathrm{deg}P} c_\alpha H_\alpha(x)$ and define $g(x) = \sum_{\alpha\leq\mathrm{deg}P} c_\alpha x_1^{\alpha_1}\cdots x_k^{\alpha_k}$. Then,
$$r(1) = \int_{\mathbb{R}^k} R\Big(\Big|\int_{\mathbb{R}^k} g(u+iv) \diff\gamma_k(v)\Big|^2\Big)\diff\gamma_k(u) \stackrel{\eqref{hermite1}\wedge\eqref{hermite2}}{=} \int_{\mathbb{R}^k} R(|P(u)|^2)\diff\gamma_k(u)$$
and, similarly, also
$$r(0) = \int_{\mathbb{R}^k} R\Big(\Big|\int_{\mathbb{R}^k} g\big(z(x+iy)\big) \diff\gamma_k(y)\Big|^2\Big)\diff\gamma_k(x)  \stackrel{\eqref{hermite1}\wedge\eqref{hermite2}}{=} \int_{\mathbb{R}^k} R(|T_zP(x)|^2)\diff\gamma_k(x).$$
Therefore, \eqref{hyp1} can be rewritten as $r(0)\leq r(1)$.

Next, we show that the monotonicity of  \eqref{izrdeba} follows from \eqref{inf1}.  Notice that for any function $Q \in C^{2}(\mathbb{R}^{k})$ such that $|Q(x)|<C(1+|x|^{N})$ for some $C, N>0$ we have 
\begin{align}\label{byparts}
\frac{\diff}{\diff s}\left(\int_{\mathbb{R}^{k}} Q(x)  \diff\gamma_k^{(s)}(x)\right) = \int_{\mathbb{R}^{k}} \frac{\Delta Q(x)}{2}  \diff\gamma_k^{(s)}(x).
\end{align}

Indeed by making change of variables $x/\sqrt{s}=y$ we obtain $ \int_{\mathbb{R}^{k}} Q(x)  \diff\gamma_k^{(s)}(x) = \int_{\mathbb{R}^{k}} Q(y\sqrt{s})  \diff\gamma_{k}(y)$. Therefore
\begin{align*}
&\frac{\diff}{\diff s}\left(\int_{\mathbb{R}^{k}} Q(x)  \diff\gamma_k^{(s)}(x)\right) = \frac{\diff}{\diff s}\left(\int_{\mathbb{R}^{k}} \frac{Q(y\sqrt{s})}{2}  \diff\gamma_{k}(y)\right)  = \int_{\mathbb{R}^{k}} \frac{\nabla Q(y\sqrt{s})\cdot y }{2\sqrt{s}}  \diff\gamma_{k}(y)
\end{align*}
and
\begin{align*}
&\int_{\mathbb{R}^{k}} \frac{\Delta Q(x)}{2}  \diff\gamma_k^{(s)}(x) = \int_{\mathbb{R}^{k}} \frac{\Delta Q(y\sqrt{s})}{2}  \diff\gamma_{k}(y).
\end{align*}
Notice that if we denote $v(y):=Q(y\sqrt{s})$  then \eqref{byparts} simply means that $\int_{\mathbb{R}^{k}} (\Delta - y\cdot \nabla )v(y) \diff \gamma_{k}(y)=0$.  The latter follows from integration by parts. Therefore, we have
\begin{align*}
&r'(s) =
\int_{\mathbb{R}^{k}}  \left[ \frac{\Delta_{u}}{2} \left( \int_{\mathbb{R}^{k}}R(|g(x,u,s)|^{2}) \diff\gamma_k^{(1-s)}(x)\right) +  \left(\int_{\mathbb{R}^{k}}R(|g(x,u,s)|^{2}) \diff\gamma_k^{(1-s)}(x)\right)_{s}\right] \diff\gamma_k^{(s)}(u).
\end{align*}
To compute the first term, one differentiation gives
\begin{align*}
\left[ \int_{\mathbb{R}^{k}}R(|g(x,u,s)|^{2}) \diff\gamma_k^{(1-s)}(x)\right]_{u_{j}} = \int_{\mathbb{R}^{k}}R'(|g(x,u,s)|^{2}) [|g(x,u,s)|^{2}]_{u_{j}}\diff\gamma_k^{(1-s)}(x),
\end{align*}
which implies that
\begin{align*}
& \left[ \int_{\mathbb{R}^{k}}R(|(x,u,s)|^{2}) \diff\gamma_k^{(1-s)}(x)\right]_{u_{j}u_{j}} =\\
&\int_{\mathbb{R}^{k}}R''(|g(x,u,s)|^{2}) ([|g(x,u,s)|^{2}]_{u_{j}})^{2}+ R'(|g(x,u,s)|^{2}) [|g(x,u,s)|^{2}]_{u_{j}u_{j}} \diff\gamma_k^{(1-s)}(x).
\end{align*}
For the second term, we have
\begin{align*}
\left(\int_{\mathbb{R}^{k}}R(|g(x,u,s)|^{2}) \diff\gamma_k^{(1-s)}(x)\right)_{s}= -\frac{1}{2} \int_{\mathbb{R}^{k}} \Delta_{x} R(|g(x,u,s)|^{2}) - 2[R(|g(x,u,s)|^{2})]_{s} \; \diff\gamma_k^{(1-s)}(x).
\end{align*}
Thus we get that $r'(s) = \frac{1}{2} \int_{\mathbb{R}^{k}} \int_{\mathbb{R}^{k}}    \tilde{r}(s)\, \diff\gamma_k^{(1-s)}(x) \diff\gamma_k^{(s)}(u)$, where 
\begin{align*}
\tilde{r}(s) =  \Delta_{u} R(|g(x,u,s)|^{2}) -\Delta_{x} R(|g(x,u,s)|^{2})+2[R(|g(x,u,s)|^{2})]_{s}.
\end{align*}
Now, compute 
\begin{align*}
&g(x,u,s)_{u_{j}} = g_{j}(x,u,s):=g_{j} \ \ \ \mbox{and} \ \ \ 
g(x,u,s)_{u_{j}u_{j}} = g_{jj}(x,u,s):=g_{jj},
\end{align*}
where $g_j$ is the $j$-th partial derivative of $g$ and we denote 
\begin{align*}
g_{j}(x,u,s) = \int_{\mathbb{R}^{k}} \int_{\mathbb{R}^{k}} g_{j}((u+iv)+z(x+iy))\diff\gamma_k^{(s)}(v) \diff\gamma_k^{(1-s)}(y),
\end{align*}
and similarly $g_{jj}(x,u,s)$ means that we first differentiate the polynomial $g$ twice in the $j$-th coordinate and then we apply {\em the flow} (\ref{flow1}).  Similarly, we have 
\begin{align*}
&g(x,u,s)_{x_{j}} = z g_{j},
\ \  g(x,u,s)_{x_{j}x_{j}} = z^{2} g_{j} \ \ \mbox{and} \ \ 
g(x,u,s)_{s} = \frac{z^{2}-1}{2}\sum_{j=1}^{k} g_{jj}.
\end{align*}
Next, further abusing the notation, we will denote $g :=g(x,u,s)$.  We have
\begin{align*}
&(|g(x,u,s)|^{2})_{u_{j}} = g_{j} \bar{g} + g\bar{g_{j}}=2  \Re \left( \frac{|g|^{2} g_{j}}{g}\right);\\
&(|g(x,u,s)|^{2})_{u_{j}u_{j}} = g_{jj} \bar{g} + g\bar{g}_{jj} + 2g_{j} \bar{g_{j}} = 2 \Re\left(\frac{|g|^{2}g_{jj}}{g}\right)+2|g_{j}|^{2};\\
&(|g(x,u,s)|^{2})_{x_{j}} = z g_{j} \bar{g} + g\bar{z}\bar{g_{j}} =2  \Re \left( \frac{|g|^{2} z g_{j}}{g}\right);\\
&(|g(x,u,s)|^{2})_{x_{j}x_{j}} = 2  \Re \left( \frac{|g|^{2} z^{2} g_{jj}}{g}\right)+2 |z|^{2}|g_{j}|^{2};\\
&(|g(x,u,s)|^{2})_{s}=  \Re \left( \frac{|g|^{2} (z^{2}-1)\sum_{j=1}^{k}g_{jj}}{g}\right).
\end{align*}

Therefore we obtain 
\begin{align*}
&\Delta_{u} R(|g(x,u,s)|^{2}) -\Delta_{x} R(|g(x,u,s)|^{2})+2[R(|g(x,u,s)|^{2})]_{s} = \\
&R'(|g(x,u,s)|^{2})\big(  \Delta_{u}(|g(x,u,s)|^{2})-\Delta_{x}(|g(x,u,s)|^{2}) +2  (|g(x,u,s)|^{2})_{s}\big)+\\
&R''(|g(x,u,s)|^{2}) \left( \sum_{j=1}^k (|g(x,u,s)|^{2})_{u_{j}}^{2} - (|g(x,u,s)|^{2})_{x_{j}}^{2} \right)=\\
&2 \sum_{j=1}^{k}\;  (1-|z|^{2}) R'(|g|^{2}) |g_{j}|^{2} + 2R''(|g|^{2})\left( (\Re g_{j} \bar{g})^{2} -(\Re z g_{j} \bar{g})^{2} \right).
\end{align*}
Notice that if $|g|=0$ then by (\ref{inf1}) we have $(1-|z|^{2})R'(0) \geq 0$, and hence $r'\geq 0$ in this case. So assume that $|g|>0$. Then, denoting $w = g_{j}\bar{g}$ and $x=|g|^2$, we see that $r'(s)\geq 0$ follows from 
\begin{align*}
(1-|z|^{2}) R'(x)|w|^{2}+2xR''(x)((\Re w)^{2}-(\Re zw)^{2})\geq 0, \quad x \geq 0, \quad w \in \mathbb{C}.
\end{align*}
The latter condition is exactly (\ref{inf1}) and the proof is complete.
\end{proof}

\begin{lemma}\label{red1}
Let $R \in C^{2}((0,\infty))\cap C([0,\infty))$, be such that $R'\geq 0$, and $R(x,\varepsilon):=R(x+\varepsilon)$ satisfies (\ref{assump1}) for each fixed $\varepsilon>0$.  Take any $z \in \mathbb{C}$, $|z|\leq 1$ such that  the inequality 
\begin{align}\label{inf2}
(1-|z|^{2}) R'(x)|w|^{2}+2xR''(x)((\Re w)^{2}-(\Re zw)^{2})\geq 0  
\end{align}
holds for all $w \in \mathbb{C}$, and all $x>0$. Then for all polynomials $P(x)$ on $\mathbb{R}^{k}$ we have 
\begin{align}\label{hyp2}
\int_{\mathbb{R}^{k}} R(|T_{z}P(x)|^{2})\diff\gamma_k(x) \leq \int_{\mathbb{R}^{k}} R(|P(x)|^{2})\diff\gamma_k(x).
\end{align}
\end{lemma}
\begin{proof}
For each $\varepsilon>0$ we consider the function $R(x,\varepsilon) :=R(x+\varepsilon)$. We claim that $R(x,\varepsilon)$ satisfies (\ref{inf1}). Indeed, applying (\ref{inf2}) at points $x+\varepsilon$, we obtain 
\begin{align*}
\frac{(1-|z|^{2})}{x+\varepsilon} R'(x+\varepsilon)|w|^{2}+2R''(x+\varepsilon)((\Re w)^{2}-(\Re zw)^{2})\geq 0.  
\end{align*}
Since $R'(x+\varepsilon)\geq 0$, $1-|z|^{2}\geq 0$, and $\frac{1}{x+\varepsilon}\leq \frac{1}{x}$ for $x>0$,  we deduce that 
\begin{align*}
\frac{(1-|z|^{2})}{x} R'(x+\varepsilon)|w|^{2}+2R''(x+\varepsilon)((\Re w)^{2}-(\Re zw)^{2})\geq 0. 
\end{align*}
The latter means that $x \mapsto R(x,\varepsilon)$ satisfies (\ref{inf1}) for all $\varepsilon >0$ (the case $x=0$ in (\ref{inf1}) is trivial because $R'(\varepsilon)\geq 0$ by the assumption in the lemma). Therefore using Lemma \ref{lem:computation}, we get 
\begin{align*}
\int_{\mathbb{R}^{k}} R(|T_{z}P(x)|^{2})\diff\gamma_k(x) \leq \int_{\mathbb{R}^{k}} R(\varepsilon+|T_{z}P(x)|^{2})\diff\gamma_k(x) \leq \int_{\mathbb{R}^{k}} R(\varepsilon+|P(x)|^{2})\diff\gamma_k(x).
\end{align*}
Next, we take $0<\varepsilon<1$ and $\varepsilon \to 0$. Notice that $\lim_{\varepsilon \to 0}R(\varepsilon+|P(x)|^{2}) = R(|P(x)|^{2})$ and   $R(\varepsilon+|P(x)|^{2})\leq R(1+|P(x)|^{2}) \in L^{1}(\mathbb{R}^{k}, \diff\gamma_k)$. Therefore we can apply Lebesgue's dominated convergence theorem and this finishes the proof of the lemma.
\end{proof}

\begin{proposition}\label{step1}
Let $B \in C([0,\infty))\cap C^{2}((0,\infty))$ be such that $B', B'' >0$. Assume $x\mapsto B(\sqrt{x+\varepsilon})$ satisfies (\ref{assump1}) for each fixed $\varepsilon>0$. Then 
\begin{align}\label{hyp3}
\int_{\mathbb{R}^{k}} B(|T_{z}P(x)|)\diff\gamma_k(x) \leq \int_{\mathbb{R}^{k}} B(|P(x)|)\diff\gamma_k(x),
\end{align}
holds for all polynomials $P(x)$ on $\mathbb{R}^{k}$, and all $k\geq 1$, if $z$ belongs to the lens 
\begin{align}\label{lenzint}
\left| 2z  \pm i \sqrt{c_{B}-2} \right| \leq \sqrt{c_{B}+2}, 
\end{align}
where $c_{B} :=\sup_{s \in (0,\infty)}\left\{ \frac{sB''(s)}{B'(s)}+\frac{B'(s)}{sB''(s)}\right\}$. 
\end{proposition}
\begin{proof}
By Lemma~\ref{red1} applied to $R(x) = B(\sqrt{x})$ we deduce that (\ref{hyp3}) holds provided that (\ref{inf2}) holds. Next, condition (\ref{inf2}) for $R(x)$ is equivalent to 
\begin{align*}
 \frac{sB''(s)}{B'(s)}\left( (\Re w)^{2} - (\Re wz)^{2}\right) +(\Im w)^{2} - (\Im wz)^{2} \geq 0,
\end{align*}
holding for all $s>0$ and $w\in \mathbb{C}$. The latter means that if we set $z=x+iy$, and $A(s) := \frac{sB''(s)}{B'(s)}$ then we must have 
\begin{align*}
\begin{pmatrix}
A(s)-A(s)x^{2}-y^{2} & xy(A(s)-1)\\
xy(A(s)-1) & 1-A(s)y^{2}-x^{2}
\end{pmatrix}\geq 0. 
\end{align*}
The trace of the matrix is $(1-x^{2}-y^{2})(A(s)+1)$, which is nonnegative if and only if $|z|\leq 1$. So it remains to study the sign of the determinant. If $y=0$ then there is nothing to check, so assume that $y\neq 0$.  The non-negativity of the determinant can be rewritten as
\begin{align*}
&A(s)+\frac{1}{A(s)} \leq \frac{(x^{2}-1)^{2}}{y^{2}}+y^{2}+2x^{2},
\end{align*}
for every $s>0$, which is equivalent to
\begin{align*}
&c_{B}-2 \leq \frac{(1-x^{2}-y^{2})^{2}}{y^{2}}
\end{align*}
and also
\begin{align*}
&|y|\sqrt{c_B-2}\leq 1-x^{2}-y^{2}.
\end{align*}
The latter inequality can be rewritten as (\ref{lenzint}). This finishes the proof of the proposition. 
\end{proof}

\subsection{Step 2. Szeg\"o Theorem}
In what follows we will be assuming that $B \in C([0,\infty))\cap C^{2}((0,\infty))$ is such that $B', B''>0$, and $x \mapsto B(\sqrt{x+\varepsilon})$ satisfies (\ref{assump1}) for each fixed $\varepsilon>0$. Next, let us consider the lens domain in $\mathbb{C}$ associated to $B$,
\begin{align}
\Omega_{B} := \left \{ z \in \mathbb{C} \; :\; \left| z \pm i \frac{\sqrt{c_{B}-2}}{2} \right| \leq \frac{\sqrt{c_{B}+2}}{2} \right\}. 
\end{align}

\begin{figure}[ht]
\centering
\includegraphics[scale=1]{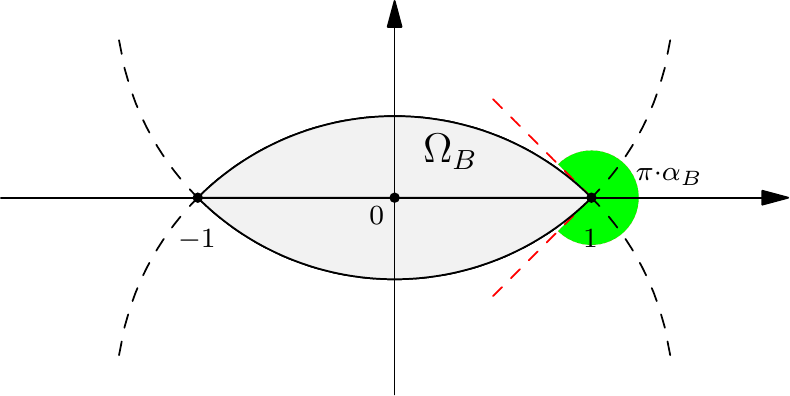}
\caption{The domain $\Omega_{B}$ and the angle $\alpha_{B} = 1+\frac{2}{\pi}\arctan\left(\frac{\sqrt{c_{B}-2}}{2}\right)$}
\label{fig:dom}
\end{figure}

The domain $\Omega_{B}$ has an exterior angle at the point $(1,0)$ which we are going to denote by $\pi \cdot \alpha_{B}$. A direct calculation reveals that  
\begin{align}
\alpha_{B} = 1+\frac{2}{\pi}\arctan\left(\frac{\sqrt{c_{B}-2}}{2}\right) \in [1,2].
\end{align}

We will need the following Markov-type inequality in the complex domain $\Omega_{B}$. For a compact set $K\subset\mathbb{C}$ and a polynomial $P$, we denote $\|P\|_{C(K)} = \sup_{z\in K}|P(z)|$ its supremum norm \mbox{in $K$.}
\begin{proposition}\label{Gabor1}
For any polynomial $P(w) = \sum_{j=0}^{n} a_{j} w^{j}$ with coefficients  $a_{j} \in \mathbb{C}$, we have 
\begin{align}\label{Mark1}
|P'(1)| \leq 10  n^{\alpha_{B}} \|P\|_{C(\Omega_{B})}.
\end{align}
\end{proposition}

Szeg\"o was the first who investigated how the geometry of a domain in the complex plane affects the growth rate of the constant in Markov's inequality. The reader can find the bound $\|P'\|_{C(\Omega_{B})} \leq C(B) n^{\alpha_{B}} \|P\|_{C(\Omega_{B})}$ in \cite{S1}, where the constant $C(B)$ depends on the domain $\Omega_{B}$. We claim here that (\ref{Mark1}) holds with a universal constant, say $C(B)=10$,  which is independent of $B$. We could not locate the proof of this claim in the literature, so we include it here for the readers' convenience.

\begin{proof}[Proof of Proposition \ref{Gabor1}] Without loss of generality assume that $n >10$, otherwise we can use the Markov inequality $|P'(1)|\leq n^{2} \| P\|_{C([-1,1])} \leq 10 n^{\alpha_{B}}  \| P\|_{C(\Omega_{B})}$. 

We map conformally $\Omega^{c}_{B}$, the complement of $\Omega_{B}$, onto $\mathbb{D}^{c}$, the complement of the unit disk, using the map
\begin{align*}
\varphi(z) = \varphi_{3}\circ\varphi_{2}\circ \varphi_{1}(z),
\end{align*}
where 
\begin{align*}
\varphi_{1}(z) = \frac{z+1}{z-1};\quad \varphi_{2}(z) = z^{1/\alpha_{B}}; \quad \varphi_{3}(z)=\frac{z+1}{z-1}.
\end{align*}
Notice that the M\"obius transformation  $\varphi_{1}(z)$  maps $\varphi_{1}(-1)=0$, $\varphi_{1}(1) = \infty$ and $\varphi_1(\infty)=1$, thus $\varphi_{1}(\Omega_{B}^c)$ is the sector centered at $z=0$ with angle $\pi \alpha_{B}$ and symmetric with respect to the positive $x$-semiaxis. Next, $\varphi_{2}(z)$ maps the sector to the right half plane $\Re z \geq 0$. Finally, $\varphi_{3}(z)$ maps the right half plane to the complement of the unit disk. It follows that 
\begin{align} \label{eq:varphi}
\varphi(z) = \frac{\left(\frac{z+1}{z-1}\right)^{1/\alpha_{B}}+1}{\left(\frac{z+1}{z-1}\right)^{1/\alpha_{B}}-1} = \alpha_{B} z + O\left(\frac{1}{z}\right), \qquad \mbox{as }  z \to \infty.
\end{align}
Next, let $P(z)$ be any polynomial of degree at most $n$ on $\Omega_{B}$ such that $\|P\|_{C(\Omega_{B})}\leq 1$ and consider the analytic function 
$$
\Omega_B^c\ni z \longmapsto  \psi(z):= \frac{P(z)}{\varphi^{n}(z)}.
$$
The function $\psi$ is regular at $z = \infty$ and bounded in absolute value by $1$ on $\partial\Omega_B$. Thus, by the maximum principle, we get
\begin{align*}
|P(z)| \leq |\varphi(z)|^{n} \qquad \text{for all} \ z \in \Omega^{c}_{B}. 
\end{align*}

Next, we will estimate $|P'(1)|$. Fix $0<\delta< \frac{1}{10}$ to be determined later and let $C_{\delta}(1)$ be the circle of radius $\delta$ centered at the point $O(1,0)$. 
Consider the arc $C_{\delta}(1) \setminus \Omega_{B}$, and let $A,B$ be its endpoints. Let $\pi\cdot \alpha'_{B} = 2\pi - \angle AOB$ where the angle $\angle AOB$ is measured in radians.  It follows from the alternate segment theorem and the law of cosines that 
\begin{equation} \label{eq:alpha'}
\alpha'_{B} = \alpha_{B} + \frac{1}{\pi}\arccos\left(1- \frac{\delta^{2}}{2R^{2}} \right) \in (1,2],
\end{equation}
where $R$ is the radius of the circles defining $\Omega_{B}$. By Cauchy's integral formula, we have 
\begin{align*}
|P'(1)| & = \left| \frac{1}{2\pi i} \oint_{C_{\delta}(1)} \frac{P(\zeta)}{(\zeta -1)^{2}} d\zeta \right| = \left| \frac{1}{2\delta} \int_{0}^{2 } \frac{P(1+\delta e^{i\pi \theta})}{e^{i \pi \theta}}  d\theta \right| \\
& \leq \frac{2-\alpha'_B}{2\delta} + \frac{ \alpha'_B }{2\delta}  \max_{\theta \in [-\alpha'_{B}/2, \alpha'_{B}/2]} | \varphi(1+\delta e^{i \pi \theta})|^{n} \leq \frac{1+\max_{\theta \in [-\alpha'_{B}/2, \alpha'_{B}/2]} | \varphi(1+\delta e^{i \pi \theta})|^{n}}{\delta}.
\end{align*}
We will need the following technical lemma.

\begin{lemma} \label{lem:technical} For all $\gamma \in [1,2]$ and $\delta, 0<\delta< \frac{1}{10}$, we have 
\begin{align}
\max_{\theta \in [-\gamma/2, \gamma/2]} | \varphi(1+\delta e^{i \pi \theta})|\leq 1+2 \delta^{1/\gamma},
\end{align}
where $\varphi$ is given by \eqref{eq:varphi}.
\end{lemma}

\begin{proof} Notice that
\begin{align*}
\varphi(1+\delta e^{i\pi\theta}) \stackrel{\eqref{eq:varphi}}{=} \frac{\left(\frac{2+\delta e^{i \pi \theta}}{\delta e^{i \pi \theta}}\right)^{1/\gamma}+1}{\left(\frac{2+\delta e^{i \pi \theta}}{\delta e^{i \pi \theta}}\right)^{1/\gamma}-1} = 1 +\frac{2\delta^{1/\gamma}}{\left(2e^{- i \pi \theta}+\delta \right)^{1/\gamma}-\delta^{1/\gamma}}.
\end{align*}
Now, since $\gamma \in [1,2]$, we see that  for every $\delta< 1/10$ we have 
\begin{align*}
|\big(2e^{- i \pi \theta}+\delta \big)^{1/\gamma}-\delta^{1/\gamma}| \geq (2-\delta)^{1/\gamma} - \delta^{1/\gamma} \geq \left(2-\frac{1}{10}\right)^{1/2} - \left(\frac{1}{10}\right)^{1/2} > 1
\end{align*}
Thus we obtain that 
\begin{align*}
\max_{\theta \in [-\gamma/2, \gamma/2]} | \varphi(1+\delta e^{i \pi \theta})| \leq 1+2\delta^{1/\gamma},
\end{align*}
which completes the proof of the lemma.
\end{proof}

\noindent {\it Proof of Proposition \ref{Gabor1} (continued).} It follows from \eqref{eq:alpha'} that since $R\geq1$, for any $0<\delta<1/10$,
$$
\alpha_B' \leq \alpha_B + \frac{1}{\pi}\arccos\left(1-\frac{\delta^{2}}{2}\right) \leq \alpha_B + \frac{11\delta}{10\pi},
$$
where the last inequality follows from the fact that $f(x):=\cos\left(\frac{11x}{10}\right) -1 +\frac{x^{2}}{2}\leq 0$ for $x \in [0,1/10]$. Next, choosing $\delta=n^{-\alpha_B}$ we thus get that
\begin{align*}
\frac{\alpha'_B}{\alpha_B} \leq   1+ \frac{1}{ \pi } \arccos\left(1-\frac{1}{2n^{2}}\right) \leq 1+ \frac{11}{10\pi n}
\end{align*}
Therefore, applying Lemma \ref{lem:technical}, the inequality $\ln (1+x) \leq x$, $x>0$, we get 
\begin{align*}
|P'(1)| & \leq n^{\alpha_{B}}\left(1+\exp\left( n \ln \left(1+\frac{2}{n^{\alpha_{B}/\alpha'_{B}}} \right)\right) \right) \leq n^{\alpha_B} \left( 1+ \exp\left(2n^{1-\frac{\alpha_{B}}{\alpha'_{B}}}\right) \right)  \\
& \leq n^{\alpha_B} \left( 1+ \exp\left(2n^{\frac{1}{1+\frac{10 \pi n}{11}}}\right) \right) \leq n^{\alpha_B} \left( 1+ \exp\left(2\cdot 11^{\frac{1}{1+10 \pi}}\right) \right) < 10n^{\alpha_{B}}.
\end{align*}
Here we have used the fact that $n \mapsto \frac{\ln n}{1+cn}$ is decreasing for $n\geq 11$ provided that $c>\frac{1}{11(\ln(11)-1)}$.
This completes the proof.
\end{proof}

\subsection{Step 3. A duality argument and the proof of Theorem \ref{mth04}}\label{bolo1} To prove Theorem \ref{mth04}, we will use a duality argument which is inspired by a similar argument of Figiel \cite[Theorem~14.6]{MS86}. 

\begin{lemma}\label{step3}
Fix a function $B\in C([0,\infty))\cap C^2((0,\infty))$. For every positive integer $n$ there exists a complex Radon measure $\diff\mu$ on $\Omega_{B}$ such that 
\begin{align} \label{eq:measure}
\int_{\Omega_{B}} z^{\ell} \diff\mu(z) = \ell, \quad \text{for all} \quad \ell=0, \ldots, n,
\end{align}
and $\int_{\Omega_B} \diff |\mu| \leq 10 n^{\alpha_{B}}$. 
\end{lemma}
\begin{proof}
Fix a positive integer $n$, and consider the functional $\psi$ on the space of polynomials of degree at most $n$ on $\Omega_{B}$, that is,  $\mathcal{P}_{n}:=\mathrm{span}_{\mathbb{C}}\{ z^{\ell}, \; \ell=0,\ldots, n\} \subset C(\Omega_B)$ given by 
\begin{align}\label{Lfun}
\psi\left( \sum_{\ell=0}^{n} a_{\ell} z^{\ell}\right) =\sum_{\ell=0}^{n}\ell a_{\ell} \quad \text{for all} \  a_{j} \in \mathbb{C}. 
\end{align}
In other words, if $P$ is a polynomial of degree at most $n$ then, $\psi(P) = P'(1)$. It follows from Proposition \ref{Gabor1} that for every such $P$, we have
\begin{align}\label{hanB1}
|\psi(P)|\leq 10 n^{\alpha_{B}} \| P\|_{C(\Omega_{B})}.
\end{align}
Therefore, by the Hahn--Banach theorem, the functional $\psi \in (\mathcal{P}_{n})^{*}$ can be extended to a functional $\Psi \in C(\Omega_{B})^{*}$ with $\|\Psi\|_{(C(\Omega_{B}))^{*}} \leq 10 n^{\alpha_{B}}$. However, by the Riesz representation theorem, the space $C(\Omega_{B})^{*}$ can be identified with the Banach space of Radon measures on $\Omega_{B}$ equipped with the total variation norm and this completes the proof of the lemma.
\end{proof}

\noindent {\it Proof of Theorem~\ref{mth04}.} Take any  complex-valued polynomial $P$ of degree at most $n$ on $\mathbb{R}^{k}$ and $z\in\Omega_B$ and consider the measure $\mu$ supported on $\Omega_B$ given by Lemma \ref{step3}. Then, we have

\begin{align*}
\int_{\mathbb{R}^{k}} B&\left( |LP(x)| \frac{1}{|\mu|(\Omega_{B})} \right) \diff\gamma_k(x)  \stackrel{\eqref{eq:measure}}{=} \int_{\mathbb{R}^{k}} B\left( \left| \int_{\Omega_{B}} T_{z}P(x)  \frac{\diff\mu(z)}{|\mu|(\Omega_{B})} \right|\right) \diff\gamma_k(x) \\
& \leq \int_{\mathbb{R}^{k}} B\left(  \int_{\Omega_{B}} \left|T_{z} P(x)\right| \, \frac{\diff|\mu|(z)}{|\mu|(\Omega_{B})} \right)\diff\gamma_k(x) \leq   \int_{\Omega_{B}} \int_{\mathbb{R}^{k}} B\left( \left| T_{z} P(x) \right|\right) \diff\gamma_k(x) \frac{\diff|\mu|(z)}{|\mu|(\Omega_{B})}  \\
& \stackrel{\eqref{hyp3}}{\leq} \int_{\Omega_{B}} \int_{\mathbb{R}^{k}} B\left( \left| P(x) \right|\right) \diff\gamma_k(x) \frac{d|\mu|(z)}{|\mu|(\Omega_{B})} =  \int_{\mathbb{R}^{k}} B\left( \left| P(x) \right|\right) \diff\gamma_k(x),
\end{align*}
where the second inequality follows from Jensen's inequality. After rescaling the coefficients of $P$ and using Lemma \ref{step3}, we deduce that the inequality
\begin{align*}
\int_{\mathbb{R}^{k}} B\left( |LP(x)|\right) \diff\gamma_k(x) \leq \int_{\mathbb{R}^{k}} B\left( 10 (\mathrm{deg}P)^{\alpha_{B}}|P(x)|\right) \diff\gamma_k(x)
\end{align*} 
holds true for all polynomials $P$ on $\mathbb{R}^{k}$. \hfill$\Box$

\medskip

We can now easily deduce Theorem \ref{mth03}.

\begin{proof} [Proof of Theorem \ref{mth03}.]
Indeed, for $p>1$ we can choose $B(s) = s^{p}$  in Theorem~\ref{mth04}. Then 
$$\frac{sB''(s)}{B'(s)}+\frac{B'(s)}{sB''(s)} = \frac{(p-1)^{2}+1}{p-1}$$
and therefore 
\begin{align*}
\alpha_{B} = 1+\frac{2}{\pi}\arctan\left(\frac{|p-2|}{2\sqrt{p-1}}\right).
\end{align*}
Then, Theorem \ref{mth04} implies that for every $p>1$,
$$\|LP\|_{L^p(\diff\gamma_k)} \leq 10 (\mathrm{deg} P)^{1+\frac{2}{\pi}\arctan\big(\frac{|p-2|}{\sqrt{p-1}}\big)} \|P\|_{L^p(\diff\gamma_k)}$$
and letting $p\to1^+$ we also deduce the endpoint case
$$\|LP\|_{L^1(\diff\gamma_k)} \leq 10 (\mathrm{deg}P)^2 \|P\|_{L^1(\diff\gamma_k)},$$
which completes the proof of the theorem.
\end{proof}


\begin{remark}
To directly prove Theorem \ref{mth03}, one could refer to the classical complex hypercontractivity ~\cite{janson}  instead of  invoking Proposition~\ref{step1} in its full generality. 
\end{remark}

\subsection{The necessity of (\ref{inf1})}\label{neces}
In addition to (\ref{assump1}) let us require that for each point $t_{0}>0$ there exists $\delta = \delta(t_{0})$ such that 
\begin{align}\label{teilor}
R(t) = R(t_{0})+R'(t_{0})(t-t_{0})+\frac{R''(t_{0})}{2}(t-t_{0})^{2} + O(|t-t_{0}|^{3})
\end{align}
holds for all $t>0$ with $|t-t_{0}|<\delta(t_{0})$.  For example if $R \in C^{3}((0,\infty))$ then (\ref{teilor}) holds. 
In particular, this means that for fixed complex numbers $a,b \in \mathbb{C}$ with $a\neq 0$ the function $t \mapsto R(|a+bt|^{2})$ has the property (\ref{teilor}). 

Fix two complex numbers $a,b \in \mathbb{C}$ with $a\neq 0$, and consider a  linear function $Q(x) = a+b \varepsilon x$ on $\mathbb{R}$, where $\varepsilon>0$. Clearly $T_{z}Q(x) = a+b\varepsilon zx$. Since for any fixed $N>0$, any polynomial $P$, and any constant $C>0$ we have $ \int_{|\varepsilon x| >C} |P(x)|\diff\gamma_1(x) = O(\varepsilon^{N})$ as $\varepsilon \to 0$  we obtain that there exists a number $\delta=\delta(a,b,z)$ such that 
\begin{align*}
&\int_{\mathbb{R}}R(|a+\varepsilon bz x|^{2}) \diff\gamma_1(x)\\
& = \int_{|\varepsilon x|\leq \delta} R(|a|^{2})+R'(|a|^{2})2 \Re (\bar{a} bz) \varepsilon x + (R''(|a|^{2})2(\Re (\bar{a} bz))^{2} + R'(|a|^{2})|bz|^{2}) |\varepsilon x|^{2} \diff\gamma_1(x) + O(\varepsilon^{3})\\
&=R(|a|^{2}) + (R''(|a|^{2})2(\Re (\bar{a} bz))^{2} + R'(|a|^{2})|bz|^{2})\varepsilon^{2} + O(\varepsilon^{3}),
\end{align*}
as $\varepsilon$ goes to zero. Similarly, there exists $\tilde{\delta} = \tilde{\delta}(a,b)$ such that 
\begin{align*}
&\int_{\mathbb{R}}R(|a+\varepsilon bx|^{2}) \diff\gamma_1(x) \\
&= \int_{|\varepsilon x|\leq \tilde{\delta}} R(|a|^{2})+R'(|a|^{2})2 \Re (\bar{a} b) \varepsilon x + (R''(|a|^{2})2 (\Re (\bar{a} b))^{2} + R'(|a|^{2})|b|^{2}) |\varepsilon x|^{2} \diff\gamma_1(x) + O(\varepsilon^{3})\\
&=R(|a|^{2}) + (R''(|a|^{2})2(\Re (\bar{a} b))^{2} + R'(|a|^{2})|b|^{2})\varepsilon^{2} + O(\varepsilon^{3}), 
\end{align*}
as $\varepsilon\to0^+$. Using (\ref{hyp2}) we thus obtain 
\begin{align*}
R''(|a|^{2})2(\Re (\bar{a} bz))^{2} + R'(|a|^{2})|bz|^{2} \leq R''(|a|^{2})2(\Re (\bar{a} b))^{2} + R'(|a|^{2})|b|^{2}.
\end{align*}
Denoting $w=\bar{a} b$ and $|a|^{2}=x>0$ the latter inequality, after multiplying the both sides by $|a|^{2}$,  takes the form 
\begin{align}\label{mokvda}
2xR''(x) (\Re (wz))^{2} + R'(x)|wz|^{2} \leq 2x R''(x)(\Re w)^{2} + R'(x)|w|^{2}.
\end{align}
Since by changing $a\neq 0$ we can make  $x$ to be an arbitrary positive number, and  by changing $b$ we can make $w$ to be an arbitrary complex number,   we see that (\ref{mokvda}) coincides with (\ref{inf1}). By continuity (\ref{mokvda}) holds also for $x=0$. This proves the equivalence between (\ref{inf1}) and (\ref{hyp1}).  \hfill$\Box$

\section{Proof of Theorem~\ref{mth02}} \label{sec4}

Recall that $L = \Delta -x\cdot \nabla$ satisfies $L H_{\alpha} = -|\alpha| H_{\alpha}$. Define $(-L)^{1/2} H_{\alpha} = |\alpha|^{1/2} H_{\alpha}$ and extend it linearly to all polynomials $P$ on $\mathbb{R}^{k}$. 
First we need the following lemma from \cite[Lemma~5.6]{LP}. Since the argument is simple, we include the proof for the readers' convenience. 

\begin{lemma}
For any $ p \geq 1$, any $k\geq 1$, and all polynomials $P$ on $\mathbb{R}^{k}$ we have 
\begin{align}\label{lustp}
\| (-L)^{1/2} P \|_{L^{p}(\diff\gamma_k)} \leq 2 \| P \|^{1/2}_{L^{p}(\diff\gamma_k)} \| L P \|^{1/2}_{L^{p}(\diff\gamma_k)}.
\end{align}
\end{lemma}
\begin{proof}
Let $C:=\int_{0}^{\infty} \frac{1-e^{-t}}{t^{3/2}} dt = 2\sqrt{\pi}$. Then for any $\lambda>0$ we have 
\begin{align*}
C \sqrt{\lambda}  = \int_{0}^{\infty} \frac{1-e^{-\lambda t}}{t^{3/2}} \diff t. 
\end{align*}
Therefore for any polynomial $P(x) = \sum_{|\alpha|\leq n} c_{\alpha} H_{\alpha}(x)$ and any number $M>0$ we have 
\begin{align*}
\big\| (-L)^{1/2} P & \big\|_{L^{p}(\diff\gamma_k)} = \frac{1}{C}\left\| \sum_{|\alpha|\leq n} c_\alpha \left( \int_{0}^{\infty} \frac{1-e^{-|\alpha| t}}{t^{3/2}} \diff t\right) H_{\alpha}  \right\|_{L^{p}(\diff\gamma_k)}\\
&=\frac{1}{C}\left\| \int_{0}^{M}  \int_{0}^{t} \left( \sum_{|\alpha|\leq n} c_\alpha  e^{-|\alpha|s}|\alpha|   H_{\alpha}\right)    \frac{\diff s \diff t}{t^{3/2}}  + \int_{M}^{\infty} \sum_{|\alpha|\leq n} c_\alpha (1-e^{-|\alpha|t})H_{\alpha} \frac{\diff t}{t^{3/2}}\right\|_{L^{p}(\diff\gamma_k)} \\
&\leq \frac{1}{C}\int_{0}^{M}\int_{0}^{t} \|T_{e^{-s}}L P\|_{L^{p}(\diff\gamma_k)} \frac{\diff s \diff t}{t^{3/2}} + \frac{1}{C} \int_{M}^{\infty} \|P-T_{e^{-t}}P\|_{L^{p}(\diff\gamma_k)} \frac{\diff t}{t^{3/2}} 
\\ & \leq \frac{2}{C}M^{1/2}\| LP\|_{L^{p}(\diff\gamma_k)} + \frac{4}{C} M^{-1/2} \|P\|_{L^{p}(\diff\gamma_k)},
\end{align*}
where we used twice the fact that the operator $T_{e^{-t}}$ is the contraction in $L^{p}(\diff\gamma_k)$ for every $t\geq0$. Finally choosing $M= 2 \|P\|_{L^{p}(\diff\gamma_k)} \|LP\|^{-1}_{L^{p}(\diff\gamma_k)}$, one arrives at \eqref{lustp}.
\end{proof}

To deduce Theorem \ref{mth02} from Theorem \ref{mth03}, we will need Meyer's Riesz transform inequalities in Gauss space ~\cite{M1} (see also \cite{MP2} for a simpler proof and \cite{G1} for a stochastic calculus approach).

\begin{theorem}[Meyer, \cite{M1}] For each $p \in (1,\infty)$  there exist finite constants $C_p, c_p>0$ such that, for any $k\geq 1$, and all polynomials $P$ on $\mathbb{R}^{k}$ we have 
\begin{align}\label{meyer1}
c_p \| (-L)^{1/2} P \|_{L^{p}(\diff\gamma_k)} \leq \|\nabla P\|_{L^{p}(\diff\gamma_k)} \leq C_p \| (-L)^{1/2} P \|_{L^{p}(\diff\gamma_k)}.
\end{align}
\end{theorem}

We can now prove Theorem \ref{mth02}.

\medskip

\noindent {\it Proof of Theorem \ref{mth02}.} Let $P$ be any polynomial on $\mathbb{R}^k$. Then, we have
\begin{align*}
\|\nabla P\|_{L^{p}(\diff\gamma_k)} & \stackrel{(\ref{meyer1})}{\leq} C_p \| (-L)^{1/2} P \|_{L^{p}(\diff\gamma_k)} \stackrel{(\ref{lustp})}{\leq}2C_p  \| P \|^{1/2}_{L^{p}(\diff\gamma_k)} \| L P \|^{1/2}_{L^{p}(\diff\gamma_k)} 
\\ & \stackrel{\eqref{Lmarkov}}{\leq} 2\sqrt{10} C_p (\mathrm{deg}\, P)^{\frac{1}{2} + \frac{1}{\pi} \arctan\big(\frac{|p-2|}{\sqrt{p-1}}\big)} \|P\|_{L^{p}(\diff\gamma_k)},
\end{align*}
and the proof is complete. \hfill$\Box$

\begin{remark} 
When $p \to \infty$ the constant $C_p$, which comes from the boundedness of the Riesz transforms \eqref{meyer1}, goes to infinity. Therefore, for large enough values of $p$ and polynomials $P$ of small enough degree, the bound \eqref{trick1} is better than \eqref{mth01}.
\end{remark}

\section{Proof of Proposition~\ref{best}}\label{mult}

In this section, we prove the sharp high dimensional $L^\infty$ version of Freud's inequality, Proposition \ref{best}. The proof is an adaptation of the argument that Freud and Nevai \cite{FrNev} have used for the real line (see also further refinements of this technique in \cite{LevLub,LevLub2}).

\medskip

\noindent {\it Proof of Proposition~\ref{best}.} Let us denote $\|f\|_{L^{\infty}(\Omega)} = \mathrm{esssup}_{x \in \Omega} |f(x)|$ and $W_{k}(x):=e^{-|x|^{2}}$, the rescaled density of the Gaussian measure on $\mathbb{R}^{k}$. Take any polynomial $P$ on $\mathbb{R}^{k}$ of degree $n$ and write $a_n=\sqrt{n/2}$. It was shown in \cite{FrNev} (see also \cite[Section~8.2]{Lub1}) that there exists a universal constant $C>0$ and a polynomial  $S_{n}$ on the real line of degree at most $Cn$ such that  
\begin{align*}
&C^{-1} W_1(x) \leq S_n(x) \leq CW_1(x),   \quad \text{for} \ |x|\leq 2a_{n}
\end{align*}
and
\begin{align*}
&|S_n'(x)| \leq C \sqrt{n} W_1(x), \quad \text{for} \ |x|\leq a_{n}.
\end{align*}
In fact, one can take $S_n$ to be the partial sums of the Taylor's series for $W_{1}(x)$ of  order $Cn$. Clearly this polynomial is even because $W_1$ is so, therefore, the function $\rho_{n}(x)=S_{n}(|x|)$ is also a polynomial on $\mathbb{R}^k$. Taking into account that $W_k(x)=W_1(|x|)$ and that $|\nabla \rho_n(x)| = |S_n'(|x|)|$, we conclude that the estimates  
\begin{align}
&C^{-1}W_k(x) \leq \rho_n(x) \leq C W_k(x)  \quad \text{for} \ x \in B(2a_{n})\label{raz}
\end{align}
and
\begin{align}
&|\nabla \rho_n(x)| \leq C \sqrt{n} W_k(x) \quad \text{on} \quad x \in B(a_n),\label{dva}
\end{align}
also hold true, where $B(r)$ denotes the closed ball of radius $r$ centered at the origin in $\mathbb{R}^{k}$. 

Next, we will need the following well-known restricted range inequality, which follows, e.g., from \cite[Theorem~1.8]{LevLub4}. For any polynomial $P$ on $\mathbb{R}^k$ of degree at most $n$, we have
\begin{align}\label{mnogom}
\|P W_{k} \|_{L^{\infty}(\mathbb{R}^{k}\setminus B(a_{n}))} \leq  \|P W_{k} \|_{L^{\infty}(B(a_{n}))}.
\end{align}
In \cite[Theorem~1.8]{LevLub4} (see also \cite[Theorem~6.2]{Lub1}), inequality \eqref{mnogom} is stated for $k=1$, i.e., for any polynomial $G$ of degree at most $n$ on $\mathbb{R}$ we have 
\begin{align}\label{odn}
\|G W_{1} \|_{L^{\infty}(\mathbb{R}\setminus [-a_{n},a_{n}])} \leq  \|G W_{1} \|_{L^{\infty}([-a_{n},a_{n}])}.
\end{align}
To deduce \eqref{mnogom}, it suffices to take an arbitrary unit vector $v$ in $\mathbb{R}^{k}$, and apply (\ref{odn}) to $G(t)=P(vt)$. Thus, it follows that for any polynomial $P$ on $\mathbb{R}^k$ of degree at most $n$,
\begin{align} \label{mnogom2}
\|P W^{2}_{k} \|_{L^{\infty}(\mathbb{R}^{k}\setminus B(a_{n}/\sqrt{2}))} \leq  \|P W^{2}_{k} \|_{L^{\infty}(B(a_{n}/\sqrt{2}))}.
\end{align}
Since $|\nabla P |^{2}$ is a polynomial of degree at most $2n$ and $a_{2n}=\sqrt{2}a_{n}$,  we have 
\begin{align*}
\||\nabla P| \; W_{k} \|_{L^{\infty}(\mathbb{R}^{k})} & \stackrel{\eqref{mnogom2}}{\leq}  \| |\nabla P| \; W_{k}\|_{L^{\infty}(B(a_{2n}/\sqrt{2}))}  =\| |\nabla P| \; W_{k}\|_{L^{\infty}(B(a_{n}))}   
 \\ & \stackrel{\eqref{raz}}{\leq} C \||\nabla P| \; \rho_{n}\|_{L^{\infty}(B(a_{n}))} \leq C(\|\nabla (P \rho_{n})\|_{L^{\infty}(B(a_{n}))} +\|P |\nabla \rho_{n}|\|_{L^{\infty}(B(a_{n}))}) \\
& \stackrel{\eqref{harris}\wedge\eqref{dva}}{\leq} C\left[ \frac{Bn}{\sqrt{n}}\|P \rho_{n}\|_{L^{\infty}(B(2a_{n}))} +C \sqrt{n}\|P W_{k}\|_{L^{\infty}(B(a_{n}))}\right] \leq A \sqrt{n} \|P W_{k}\|_{L^{\infty}(\mathbb{R}^{k})},
\end{align*}
for some universal constants $A,B>0$. Here, we also used multidimensional Bernstein inequality 
\begin{equation} \label{harris}
\|\nabla P\|_{L^{\infty}(B(R))} \leq   \frac{Bd}{R} \|P\|_{L^{\infty}(B(2R))},
\end{equation}
of Harris \cite{Harris} (see also ~\cite{SAR}), where $B>0$ is a universal constant. Finally making the change of variables $x = y/\sqrt{2}$, and dividing of both sides of the one but last inequality by $\sqrt{(2\pi)^{k}}$ we obtain the estimate
\begin{align*}
\left\| |\nabla P(x)| \frac{e^{-|x|^{2}/2}}{\sqrt{(2\pi)^{k}}}\right\|_{L^{\infty}(\mathbb{R}^{k})} \leq C \sqrt{n} \left\| P(x) \frac{e^{-|x|^{2}/2}}{\sqrt{(2\pi)^{k}}}\right\|_{L^{\infty}(\mathbb{R}^{k})}
\end{align*}
for a universal constant $C>0$ and all polynomials $P$ on $\mathbb{R}^k$ of degree at most $n$. This finishes the proof of Proposition~\ref{best}. \hfill$\Box$

\end{document}